\documentclass[11pt,a4paper,reqno]{amsart}

\usepackage{a4wide,color}

\usepackage[english]{babel}
\usepackage[utf8]{inputenc}

\usepackage{bbm,amsmath,amsthm,epsfig,latexsym,marvosym}
\usepackage{amsfonts}
\usepackage{bigints}
\usepackage{amsmath}
\usepackage{amssymb}
\usepackage{subfig}

\allowdisplaybreaks

\usepackage{esint,xfrac}
\usepackage[colorlinks=true,linkcolor=blue,citecolor=red]{hyperref}

\newcommand{\mres}{\mathbin{\vrule height 1.6ex depth 0pt width
0.13ex\vrule height 0.13ex depth 0pt width 1.3ex}}

\def\A{\mathbb A}

\def\R{\mathbb R}
\def\N{\mathbb N}
\def\C{\mathcal C}

\def\cal{\mathcal}

\def\AA{{\mathcal A}}

\def\E{{\cal E}}

\def\H{{\cal H}}

\def\L{{\cal L}}

\def\CC{{\mathbb C}}
\def\a{\alpha}

\def\e{\varepsilon}

\def\d{\mathrm{d}}
\def\div{\mathrm{div}}

\def\Om{\Omega}

\def\ov{\overline}

\def\Tan{\mathrm{Tan}}
{\left\lbrace\begin{array}{@{}l@{}}}%
{\end{array}\right.}
\def\BDdev{\mathrm{BD}_{\mathrm{dev}}}
\def\BD{\mathrm{BD}}
\def\BV{\mathrm{BV}}

\def\wt{{\rightharpoonup^*}\,}
\def\Id{{\rm Id}\,}
\def\loc{{\rm loc}}

\def\d{\, \mathrm{d}}

\def\dist{{\rm dist}}
\def\curl{{\rm curl}}

\def\00{{\bf 0}}
\def\dive{{\rm div}}

\DeclareMathOperator*{\spt}{spt}

\def\big{\bigskip}

\newtheorem{theorem}{Theorem}[section]

\newtheorem{proposition}[theorem]{Proposition}
\newtheorem{lemma}[theorem]{Lemma}

\newtheorem*{theorem*}{Theorem}

\theoremstyle{definition}
\newtheorem{remark}[theorem]{Remark}
\newtheorem{definition}[theorem]{Definition}

\numberwithin{equation}{section}
\numberwithin{figure}{section}

\usepackage{color}

\title[Iterative blow-ups technique: a refinement]
{Iterative blow-ups for maps with bounded $\AA$-variation: a refinement, with application to $\BD$ and $\BV$.}

\author{Marco Caroccia}
\address{Dipartimento di matematica, Politecnico di Milano}
\email{marco.caroccia@polimi.it}

\author{Nicolas Van Goethem}
\address{Centro de Matemática e Aplicações Fundamentais, UniversI.de de Lisboa}

\email{vangoeth@fc.ul.pt}

\makeindex

\begin{document}
\maketitle
\begin{abstract}
    We refine the iterated blow-up techniques. This technique, combined with a rigidity result and a specific choice of the kernel projection in the Poincaré inequality, might be employed to completely linearize blow-ups along at least one sequence. We show how to implement such argument by applying it to derive affine blow-up limits for $\BD$ and $\BV$ functions around Cantor points. In doing so we identify a specific subset of points - called totally singular points having blow-ups with completely singular gradient measure $D p=D^s p$, $\E p=\E^s p$ - at which such linearization fails.
\end{abstract}
\tableofcontents
\section{Introduction}

In this paper we consolidate the iterated blow-up approach that has been developed in \cite{caroccia2019integral} based on the work \cite{de2019fine}. This approach consists of iteratively blowing up a function at a point, relying on the principle expressed in \cite[Theorem~14.16 ]{mattila1999geometry}, which can be summarized by Preiss' result that \textit{tangent measures to tangent measures are tangent measures} (see also Theorem \ref{thm:Mattila} below).  In other words, if $h$ is a blow-up of $u$ at $x$, and $g$ is a blow-up of $h$ at $y$, then, by means of this principle, we can deduce that $g$ must be a blow-up of $u$ at $x$.\\

This idea has already been successfully implemented in \cite{de2019fine} to obtain relaxation and in \cite{caroccia2019integral} to obtain integral representation results for variational functionals in the context of maps of bounded deformation in the spirit of \cite{bouchitte1998global}. For applications, see also \cite{ap3,ap2,ap0,ap1}. The aim of this paper is to establish a general framework in order to apply this technique to general first-order operators $\AA$. Note that the original technique introduced in \cite{de2019fine} deals with iterative blow-ups of the measure $\AA u$ ($\E u$ in that case), slightly different from the approach proposed here, which considers blow-ups of the function $u$. The main difference lies in the following fact: by considering blow-ups of $\AA u$ one can obtain relaxation and homogenization results for energies depending solely on $\AA u$. By focusing on the blow-ups of $u$, instead, it is possible to obtain relaxation and integral representation for energies depending on $\AA u$ \textbf{and} $u$, up to some extent (as done in \cite{caroccia2019integral}). To obtain this slightly more general result, some additional ingredients are, however, required: a rigidity result and the structure of the operator  \textit{projection onto the kernel} $\mathcal{R}:L_{loc}^1(\R^n;\R^m)\rightarrow \mathrm{Ker}(\AA)$, appearing in the Poincaré inequality (see Proposition \ref{prop:rindlerchar} and Lemma \ref{l:auxiliary} in the context of $\BD$). By rigidity result, we mean some general structure property implied by having constant polar $\frac{\d \AA u}{\d |\AA u|}$ on the convex set $K$. In some sense, since we need to keep track of the pointwise value of $u$ (to allow for energies depending on $\AA u$ and $u$) it is no surprise that we must gain some information on $u(x)$, encoded as information on the kernel projection $\mathcal{R}$, and the rigidity of blow-up structure at $x$. Once these two ingredients are at hand heuristically, by iteratively and alternatively blowing up, using rigidity and the kernel structure, we can find at least one affine  blow-up of $u$ at $|\AA u|$-a.e. point, out of a specific set identified as the totally singular points (cf. Definition  \ref{def:TSpoinBV}, \ref{def:TSpoinBD}). We show how to implement this scheme at Cantor points, in the $\BV$ and $\BD$ case as a refinement of the technique developed in \cite{caroccia2019integral} in Section \ref{sct:app}.  
In particular by relying upon Definition \ref{def:Strict} and Proposition \ref{lem:BB}, which embed the idea of iterative blow-ups into a solid general framework, we prove the following two facts. 
\begin{theorem*}[Affine blow-ups for $\BV$ functions - Theorem \ref{thm:mainBV}]
Let $n\geq 2$ and $u\in \BV(\Omega;\R^m)$. Let $K$ be a center-symmetric convex set.  Then for $|D^c u|$-a.e. $x\in \Omega\setminus \mathrm{TS}(u)$ there exists a vanishing sequence $ \e_i \downarrow 0 $ such that
 	\[
	u_{K,\e_i,x}(y)\rightarrow  \frac{\d D u}{\d |D u| } (x) y \ \ \text{strictly in $\BV(K;\R^n)$}.
	\]
\end{theorem*}  
\begin{theorem*}[Affine blow-up for $\BD$ functions - Theorem \ref{thm:blupSel}]\nonumber
    Let $u\in\BD(\Omega)$. Let $K$ be a center-symmetric convex set. Then for $|\E^c u|$-a.e. $x\in\Omega\setminus \mathrm{TS}(u)$ there exists a sequence $\e_i \downarrow 0$ such that
\[
u_{K,\e_i,x}(y)\rightarrow \frac{\d \E u}{\d |\E u|}(x)y \ \ \text{strictly in $\BD(K)$ }
\]
\end{theorem*}

 Here, $u_{K,x,\e_i}$ is the typical blow-up sequence considered when dealing with homogenization problems
 \[
 u_{K,\e,x}(y):=\frac{u(x+\e y) -\mathcal{R}_K[u(x+\e \cdot)](y)}{\frac{|\mathcal{A} u|(K_{\e}(x) )}{|K|\e^{n-1}}},
 \]
 where $\AA=\E$ in the $\BD$ case, $\AA=D$ in the $\BV$ case, and $\mathcal{R}_K$ is a specific linear bounded operator (related to $K$), mapping $L^1$ onto $\mathrm{Ker}(\AA)$. The set $\mathrm{TS}(u)$ is a specific set of points, called \textit{totally singular points}, defined for $\BV$ (Definition \ref{def:TSpoinBV}) and for  $\BD$ (Definition \ref{def:TSpoinBD}) as those points for which all the blow-ups are given by functions $h$  having only the singular part in their gradient (i.e. $D^s h = D h$ or $\E h=\E^s h$ for the symmetric gradient). The exclusion of these points is due to the following fact. Heuristically, we eliminate the singular part of $\E u$ ($Du$ in the $\BV$ case) in the final blow-up by iteratively blowing up at absolutely continuous points of each new blow-up, ending with an affine function. If we try to apply this idea at a totally singular point, after the first blow-up we can no longer find a point of absolute continuity where $\psi'(x)\neq 0$ to perform a second blow-up (since $D\psi = D^s\psi$). While a complete linearization is possible only outside of $\mathrm{TS}(u)$, we underline the general statements \cite{de2019fine}, \cite[Proposition 4.6]{caroccia2019integral}  for $\BD$ that allow to linearize at least one direction for $|\E^c u|$-a.e. $x\in \Omega$. Thus, at totally singular points, the process must stop, and we are left with one-dimensional blow-ups with a vanishing absolutely continuous part. Of course further blow-ups can follow, with a suitable application of Proposition \ref{lem:BB} if, for the blow-up $h$, it holds $e(h)\L^n\neq 0$. We also point out that in \cite{caroccia2019integral} the linearization is achieved for specific convex sets $P$ but with this refined technology it can be checked that \cite[Proposition 4.6]{caroccia2019integral} can be proven for any general $K$. \\  

As we briefly introduce at the beginning of Section \ref{sct:app}, we consider the iterative blow-up strategy to be a fruitful and robust scheme for tackling homogenization and relaxation challenges.\\

The paper is organized as follows. In Section \ref{sct:Prel} we introduce the main notations and we retrieve, in a survey-like presentation, the results in literature on general first-order differential operator $\AA$ and the space of bounded $\AA$-variation. In Section \ref{sct:it}, we define and elaborate on the set of \textit{strict blow-up limit} (Definition \ref{def:Strict}) $\mathrm{bu}_K(u,x)$, we prove that is never empty and that \textit{strict blow-ups of blow-ups are strict blow-ups} (i.e. Proposition \ref{lem:BB}). Finally in Section \ref{sct:app}, we apply the iterative blow-up technique to obtain affine blow-ups for $\BD$ and $\BV$ functions. 

\section{Preliminaries}\label{sct:Prel}
We collect here some preliminary results from literature that will be used in the sequel, together with the setting of the notation used in rest of the paper.
\subsection{General notations}
The letter $n$ will always denote the ambient Euclidean space dimension. We will denote by $B_r(x)$ the ball of radius $r$ and centered at $x$. Whenever $x=0$ we just write $B_r$, as well as in the case $r=1$ when we simply write $B(x)$. More in general, given a convex body $K$ we denote by $K_{r}(x):=x+rK$. We denote by $\mathbb{M}^{m\times n}$ the set of $n\times n$ matrices. The notation $e_i$ stands for the $i$-th vector of the canonical basis of $\R^n$, $\Id$ denotes the $n\times n$ identity matrix. The notation $\L^n$, $\H^{n-1}$ stand for the $n$-dimensional Lebesgue measure and the $(n-1)$-dimensional Hausdorff measure on $\R^n$ while $\mathcal{M}(E;V)$ is the space of all finite $V$-valued Radon measures on $E$. The space $\text{Lin}(X;Y)$ denotes the family of all linear maps between the two  vector spaces $X$ and $Y$. \\


  Let then $\{A_j\}_{j=1}^n \subset \mathrm {Lin}(\R^m,V)$, for some Euclidean space $V$, and define the \textit{first order linear operator with constant coefficient} $\AA: C^1(\R^n;\R^m) \rightarrow C^0(\R^n; V)$ to be:
  \begin{align*}
   \AA u(x):=\sum_{j=1}^n A_j \partial_j u (x).
  \end{align*}
To such an operator, and for each $\xi\in \R^n$, we associate the \textit{symbol} $\A[\xi]:  \R^m \rightarrow V$
\begin{align}
   \A[\xi]\eta :=\sum_{j=1}^n \xi_j A_j \eta.\label{AAdef}
  \end{align}
  We may use the intuitive notation introduced in \cite{breit2017traces}:
  $$
   \A[\xi]\eta=\eta\otimes_\A \xi,
  $$
  where we have defined the bi-linear map $\otimes_\A(\eta,\xi)=\eta\otimes_\A\xi=\A[\xi]\eta$. Observe that, formally for $u\in C^1(\R^n;\R^m) $ and since $D=(\partial_1, \ldots,\partial_n)$,
\[
\AA u(x)= \A[D]u.
\]
  \subsection{Maps of bounded $\AA$-variation}
Now we define the set of functions with bounded $\AA$-variation as the functional space (still adopting the notation in \cite{breit2017traces})
    \[
   \BV^{\AA}(\Omega;\R^m):= \left\{\left. u\in L^1(\Omega;\R^m) \ \right| \ \AA u\in \mathcal{M}(\Omega;V)\right\},
    \]
    where $\mathcal{M}(\Omega;V)$ is the set of $V$-valued finite Radon measures (see \cite{ambrosio2000functions}), and with the distributional $\AA$-gradient defined as the following measure:
    \[
    \int_\Om\varphi \cdot \d\AA u:=\int_\Om \AA^*\varphi\cdot u \d x, \ \forall\  \varphi\in C^{\infty}_c(\Om;V),
    \]
    where the $L^2$-adjoint operator $\AA^*:C^1(\R^n;V)\rightarrow C^0(\R^n;\R^m)$ is defined, starting from $\{A_j^*\}_{j=1}^n \subset \mathrm{Lin}(V;\R^m)$, and with $A_j z \cdot \eta =z \cdot A_j^* \eta$ , for all $z\in \R^m,\eta \in V$, as 
    \[
    \AA^*v=-\sum_{j=1}^n A^*_j \partial_jv .
    \]
    
 As classical examples, the spaces of function of bounded variations $\BV(\R^n;\R^m)=\BV^D(\R^n;\R^m)$ (with $\AA u=Du=\mathbb A[D]u=u\otimes D$; $\AA^*\varphi=-\div\varphi$ and $V=\mathbb R^{m\times n}$), the spaces of function of \textit{bounded deformations} $\BD(\R^n)=\BV^{\E}(\R^n;\R^n)$ with 
 \[
\E u= \frac{1}{2}\left(Du+D^Tu\right)
 \]
 (with $\E  u=\mathbb E[D]u:=u\otimes_\E D$;  $\E^*\varphi=-\div\varphi$  and $V=\mathbb R^{n\times n}_\mathrm{sym}$)
and the space of \textit{bounded deviatoric deformations} $\BDdev(\R^n)=\BV^{\E_d}(\R^n;\R^n)$ with
\[ 
\E_d u= \E u-\frac{\dive(u)}{n} \Id
\]
(with $\E_d u=\mathbb E_d[D]u:=u\otimes_{\E_d} D $; $\E_d^*\varphi=-\div\varphi$ and $V=\mathbb R^{n\times n}_{\mathrm{sym}_0}$).
\smallskip
For references about these spaces we refer to \cite{ambrosio1997fine,ambrosio2000functions,arroyo2019fine,ebo99b,kohn1980new, rindler2018calculus,TS2}. \\

Let us now introduce some crucial concepts in order to deepen the properties of these operators.

\subsubsection{Ellipticity and Cancelling properties}
\begin{definition}[Elliptic]
We say that the operator $\AA$ is \textit{elliptic} (or $\R$-elliptic) if for any $\xi\in \R^n\setminus \{0\}$ the map $\A[\xi]: \R^m \rightarrow V$ is one-to-one. Equivalently, if and only if for all non-zero $\xi$ there exists a constant $c=c(n,\AA)>0$ such that 
\begin{align}\label{elliptic}
   |\A[\xi]\eta|\geq c|\xi||\eta|, \quad \forall \xi\in \R^n, \eta\in V.
\end{align} 
\end{definition}

Intuitively ellipticity means that the equation $\AA u=f$ is left-invertible (i.e., has a unique solution) in Fourier space \cite{gmeineder2017critical}. As a consequence of the celebrated work by Calderón and Zygmund \cite{Calderon1952} we have that $\AA$ is ($\R$-)elliptic if and only if for each $1<p<\infty$ there exists $c=c(p,n,\AA)$ such that coercivity is in force:
\begin{align}\label{elliptic2}
\|Du\|_{L^p(\R^n;\mathbb M^{m\times n})}\leq c\|\AA u\|_{L^p(\R^n;V)},
\end{align}
for all $u\in C_c^\infty(\Om;V)$.\\

Inequality \eqref{elliptic2} is well-known to be false for $p=1$ (cf. \cite{ornstein1962non}) but under a stronger assumptions on $\AA$  we can infer a Poincaré-type inequality and a Sobolev-type inequality. 

\begin{definition}[Cancelling]
We say that $\AA$ is cancelling if
    \begin{align}
  \bigcap_{\xi \in \R^n\setminus\{0\} } 
\mathrm{Im}(\A[\xi])=\{
0\}.\label{cancellation}
\end{align}
\end{definition}
To the knowledge of the authors, the first contribution on the $L^p$-differentiability results of $BV^\mathcal A$-maps under ellipticity and
cancellation (not necessarily $\CC$-elliptic) properties can be found in \cite{Raita1}.

For an elliptic and canceling first order linear operator  we have the following Theorem.
\begin{theorem}[Gagliardo-Nirenberg-Sobolev, see \cite{vanSchaftingen2013}]\label{thm:GaNiSob}
  If $\AA$ is an elliptic and canceling first-order linear operator then there exists a $c>0$, depending on $n$ and $\A$ only, such that  
  \[
\|u\|_{L^{\frac{n}{n-1}}}\leq c \|\AA u\|_{L^1 } \ \ \ \text{for all $u\in C^{\infty}_c(\R^n;\R^m)$}.
\]
\end{theorem}        
It is not easy to localize Theorem \ref{thm:GaNiSob} to general bounded open set. In this sense the best result available in literature (see Proposition \ref{prop:sobBalls}) requires the introduction of the $\CC$-ellipticity, a strong property that allows to develop all main variational tools on $\BV^\AA$.

\begin{definition}[$\mathbb{C}$-Elliptic]
We say that the operator $\AA$ is $\mathbb{C}$-\textit{elliptic} if for any $\xi\in \mathbb{C}^n\setminus \{0\}$ the complexification of its symbol $\A[\xi]: \mathbb{C}^m \rightarrow V+iV$ (being $i$ the imaginary unit) is one-to-one. 
\end{definition}

Obviously a $\mathbb{C}$-elliptic operator is also elliptic.  In \cite[Theorem 2.6]{breit2017traces} is shown that $\CC$-ellipticity is equivalent to the property of finite-dimensionality of the kernel of $\AA$, i.e. $\mathrm{dim}\{v\in\mathcal D'(\Omega;V): \AA v=0\}<\infty$. This makes easy to verify, for instance, that $D,\E$ are $\CC$-elliptic as well as to verify that $\E_d$ is \textbf{not} $\CC$-elliptic in $n=2$ (whereas it is in bigger dimension).  \\

As a matter of fact, many important operators are not $\CC$-elliptic, as for instance the curl (of order $1$) or the incompatibility operator (of order $2$), since their kernels, consisting of all gradients of scalar function (for the curl), and of all symmetric gradients of vector functions (for the incompatibility \cite{MSVG2015}), have not finite dimension. However note that operators $\curl$ and $\mathrm{inc}$ may verify  elliptic-like properties, in the sense that \eqref{elliptic2} holds for some specific Sobolev spaces $V$ (cf. \cite[Lemma 7]{SVG2016} for the curl and \cite[Theorem 3.9]{AVG2016} for the incompatibility).\\

Finally, $\CC$-ellipticity allows one to localize Theorem \ref{thm:GaNiSob} to obtain the following Poincaré-Sobolev inequality. In the following, $\Pi^{U}_{\mathrm{Ker} }:L^1(U;\R^m)\rightarrow \mathrm{Ker}(\mathcal{A})$ stands for a bounded linear projection operator onto the kernel of $\AA$, denoted as $\mathrm{Ker}(\mathcal{A})$.
\begin{proposition}[Poincaré-Sobolev inequality]\label{prop:sobBalls}
Let $\AA$ be a $\CC$-elliptic first order differential operator with constant coefficients. Let $K$ be a center-symmetric convex set. Then there exists a constant $c$ depending on $n$ and $K$ only such that
\begin{align}\label{SGNforA}
      \|u-\Pi^{K_r(x)}_{\mathrm{Ker} }u\|_{L^\frac{n}{n-1}(K_r(x);\R^m)}\leq c |\AA u|(\overline{K_r(x)}) \ \ \ 
\end{align}
for all $x\in \R^d$, $r>0$ and $u\in \BV_{loc}^\AA(\R^n;\R^m)$. 
\end{proposition}
\begin{remark}
 We underline that the result of Proposition  \ref{prop:sobBalls} is present in literature, in \cite[Proposition 2.5]{gmeineder2017critical} in the case $K=B$. By carefully looking at the proof, the argument relies upon the extension operator $E: W^{\mathcal{A},1}(B_r(x))\rightarrow W^{\mathcal{A},1}(\R^d)$ built in \cite{GmeinRaita2019}. However, the argument for building the extension operator in \cite{GmeinRaita2019} is not sensible to the shape of the ball and thus it can be applied verbatim to produce an extension operator $E: W^{\mathcal{A},1}(K_r(x))\rightarrow W^{\mathcal{A},1}(\R^d)$ for generic convex sets $K$ (in fact on every so-called Jones domain). The same proof as in \cite[Proposition 2.5]{gmeineder2017critical} yields the result of Proposition \ref{prop:sobBalls} on generic convex sets $K$. We are deeply grateful to F. Gmeineder for the fruitful clarification about this subject.
\end{remark}
The space $\BV^\AA(\Omega)$, endowed with the norm $\|u\|_{\BV^{\AA}}:=|\AA u|(\Omega)+\|u\|_{L^1}$, is a Banach space. The Poincaré-Sobolev inequality in \eqref{prop:sobBalls} provide a standard argument, by following for instance the ideas in \cite{kohn1980new}, to prove the following compactness Theorem.
\begin{theorem}[Compactness Theorem]\label{thm:cmp}
Let $\Omega\subset \R^n$ be an open bounded set with Lipschitz boundary and $\AA$ be a $\C$-elliptic first-order linear operator. Let $\{u_k\}_{k\in \N}\subset \BV^\AA(\Omega;\R^m)$. Suppose that
\[
\sup_{k\in \N}\{\|u_k\|_{\BV^{\AA}}\}<+\infty.
\]
Then there exists $u\in \BV^\AA(\Omega)$ and a subsequence $h(k)$ such that $u_{h(k)}\rightarrow u$ in $L^1$ and $\AA u_{h(k)} \wt \AA u$.
\end{theorem}

The notation $\AA u_{h(k)} \wt \AA u$ stands for the standard \textit{weak$^*$ convergence} of Radon measures (see \cite{Maggi} or \cite{evans2018measure}).\\

 \subsubsection{Poincaré  inequality for $\CC$-elliptic operators}
 
 Let $\AA$ be a linear, ﬁrst order, homogeneous diﬀerential operator with constant coeﬃcients on  $\R^n$ which is $\CC$-elliptic. Let $K$ be a convex set of $\R^n$.
Then it is known, \cite[Theorem 3.7]{diening2021sharp}, that for a  uniform constant $c=c(K,n)>0$ the following Poincaré-type of inequality holds:
\begin{align} 
      \|u-\Pi^K_{\mathrm{Ker} }u\|_{L^{1}(K;\R^m)}\leq c|\AA u|(K).
\end{align}
Since we need to apply Poincaré inequality  (for compactness purposes) on specific projection operators the following Proposition might be of some use. The proof is exactly as in \cite{breit2017traces}, retrieved here for the sake of completeness.

\begin{proposition}[Poincaré inequality]\label{prop:poincare}
Let $K$ be a fixed convex set. Let $\AA$ be $\CC$-elliptic and let $\mathcal{R}: L^{1}(K;\R^m)\rightarrow \mathrm{Ker}(\mathcal{A})$ be a linear, bounded operator such that $\mathcal{R}(L)=L$ for all $L\in\mathrm{Ker}(\mathcal{A}) $. Then there exists a uniform constant $c=c(\mathcal{R},K,n)$ depending on $n,\mathcal{R}$ and $K$ such that
    \begin{align}\label{Poincare}
    \|u-\mathcal{R}[u]\|_{L^{1}(K;\R^m)}\leq  c|\AA u|(K).
    \end{align}
\end{proposition}
\begin{proof}
Since $\AA$ is $\CC$-elliptic, inequality \eqref{SGNforA} holds. Then
   \begin{align*}
    \|u-\mathcal{R}[u]\|_{L^{1}(K;\R^m)}&\leq\|u-\Pi_{\mathrm{Ker} }^K u\|_{L^{1}(K;\R^m)}+\|\mathcal{R}[u]-\Pi_{\mathrm{Ker} }^Ku\|_{L^{1}(K;\R^m)}\\
    &\leq c |\AA u|(K)+\|\mathcal{R}[u-\Pi_{\mathrm{Ker} }^Ku]\|_{L^{1}(K;\R^m)}
   \end{align*}
   Since $\mathcal{R}$ is continuous we also have
    \begin{align*}
    \|\mathcal{R}(u-\Pi_{\mathrm{Ker} }^K u)\|_{L^{1}(K;\R^m)}&\leq c(\mathcal{R},K,n) \|u-\Pi_{\mathrm{Ker} }^Ku\|_{L^{1}(K;\R^m)}\leq c|\AA u|(K)
    \end{align*}
thus proving \eqref{Poincare}.
\end{proof}

\subsection{Structure of maps of bounded $\AA$-variation maps}\label{sbs:structure}
We here collect the results related to the structure of $\AA u$ under different assumptions on $\AA$. The majority of these results heavily rely upon the pioneering works \cite{arroyo2019dimensional},\cite{breit2017traces} and \cite{de2016structure}.

\subsubsection{Trace and Gauss-Green theorems}
Under $\CC$-ellipticity, in \cite{breit2017traces} is shown that a function $u\in\BV^{\AA}(\Omega;\R^m)$ possesses the trace on $\partial\Omega$ and the Gauss-Green formula
\[
\int_{\Omega} \varphi(x)\cdot \d \AA u(x) = \int_{\Omega} \AA^*\varphi (x)\cdot  u(x)\d x+\int_{\partial \Omega} (u(y)\otimes_{\A} \nu_\Omega(y) )\cdot \varphi(y) \d \H^{n-1}(y)
\] 
holds true for all $\varphi\in C^{\infty}(\Omega;V)$. The trace is continuous under \textit{strict convergence}: namely if $u_k\rightarrow u$ in $L^1$ and $|\AA u_k|(\Omega) \rightarrow |\AA u|(\Omega)$ then the trace $u_k\Big{|}_{\partial \Omega}\rightarrow u\Big{|}_{\partial \Omega}$ in $L^1$.  
To the knowledge of the authors, the first 
trace inequalities on $s\in (n-1, n)$-dimensional sets can be found in \cite{Gmeineder1}. 

 \begin{definition}[Approximate jump\label{def:jump}] Let $u \in \
L^1_\loc(\Omega;\R^m)$.
        We say that a point $x$ is an \emph{approximate jump point} of $u$ if there exist \emph{distinct} vectors $u^+,u^- \in \R^m$ and a direction $\nu \in \mathbb{S}^{n-1}$ satisfying 
        \begin{equation}\label{eq:jumps}
       \left\{ \begin{array}{l}     
     \displaystyle
        \lim_{r \to  0} \fint_{B^+_r(x,\nu)} |u(y) - u^+| \d y = 0,\\
        \displaystyle
        \lim_{r \to 0} \fint_{B^-_r(x,\nu)}  |u(y) - u^-| \d y = 0.
   \end{array} \right.
        \end{equation}
        Here, we use the notation
        \[
        B^+_r(x,\nu):=\{ y \in B_r(x) \  \nu\cdot (x-y)  > 0\},\qquad   B^-_r(x,\nu):=\{ y \in B_r(x) \  \nu\cdot (x-y)  < 0\}
        \]
        for the $\nu$\emph{-oriented half-balls} centred at $x$. We refer to $u^+,u^-$ as the \emph{one-sided limits} of $u$ at $x$ with respect to $\nu$. The triplet is uniquely idendified up to a change of sign for $\nu$ and a permutation of $u^+,u^-$. The collection of all points of approximate jump for $u$ is denoted as $J_u$.
\end{definition}
We underline the remarkable result in \cite{del2021rectifiability} proving that, for a function $u\in L^1_{loc}(\Omega;\R^m)$ (without any additional information on the derivatives of $u$), $J_u$ is $\H^{n-1}$-rectifiable.

To the knowledge of the authors, the first characterisation of continuity points, and first
contribution on the fine properties of $BV^\mathcal A$-maps can be found in \cite{Diening1}. Moreover, the first systematic understanding of the algebraic properties of symbols to
yield trace theorems and hereafter fine properties on lower dimensional
sets can be found in \cite{Gmeineder2}. 
\subsubsection{Radon-Nikodým decomposition of $\AA u$}
Let $u\in\BV^{\AA}(\Omega;\R^m)$ and let us denote by $S_u$ the complement of the set of points of approximate continuity of $u$ (those points where $u^+=u^-$). Then Radon-Nikodým writes as
$$
\AA u=\AA^a u+\AA^su=\AA^a u+\underset{=\AA^{c}}{\underbrace{\AA^su\mres(\Om\setminus S_u)}}+\underset{=\AA^{dd}}{\underbrace{\AA^su\mres(S_u\setminus J_u)}}+\underset{=\AA^{j}}{\underbrace{\AA^su\mres J_u}}.
$$
The term $\AA^{c}$ is the \textit{diffuse continuous or Cantor-singular part}. The term $\AA^{dd}$ is the \textit{diffuse discontinuous singular part}. \\

It is proved in \cite[Lemma 3.1]{gmeineder2017critical} that, under $\R$-ellipticity of $\AA$, $u\in \BV^{\AA}(\Om;\R^m)$ is $L^p$-differentiable (and hence approximately-differentiable) $\L^n$-almost everywhere for any $1\leq p<\frac{n}{n-1}$. For $p=\frac{n}{n-1}$ the same holds but $\CC$-ellipticity is needed for $\AA$.  Finally, for $\R$-elliptic operator (see \cite[Theorem 3.4.]{alberti2014p}) we have $\L^n$-almost everywhere that
\begin{equation}\label{eqn:appdiff}
\frac{d\AA u}{d\mathcal L^n}(x)\L^n=\A[\nabla]u(x),
\end{equation}
yielding $\AA^a u= \A[\nabla]u \L^n$.\\

For $\CC$-elliptic operators it is known \cite[Theorem 1.2]{arroyo2019fine}, that $(S_u\setminus J_u)$ is purely $\H^{n-1}$-unrectifiable and that the jump parts writes   as \cite{breit2017traces}
$$ \AA^{j}=(u^+-u^-)\otimes_\A \nu \H^{n-1}\mres J_u.  $$

Finally, for first order linear operators $\AA$ satisfying a specific hypothesis called $\textit{rank}_\mathcal{A}$-\textit{one} property, in \cite{arroyo2020slicing} it is proven that $|\AA u|(S_u\setminus J_u)=0$ implying $\AA^{dd}=0$.\\
 Moreover,  according to \cite{arroyo2020slicing}, the rank-one property allows one to perform one-dimensional slicing of such maps, i.e., there exist directions $a,b$ such that $\forall v\in V$ it holds
\[
 v\cdot\AA u=\int_{\pi_a}Du^b_{y,a}d\H^{n-1}(y),
 \]
 where $\cdot$ is by Riesz intended for the duality $(V,V^*)$,
 with the slice of $u$ defined as $u^b_{y,a}:\Om^a_y:=\{s\in\R:y+sa\in\Om\}\to\R$, where $u^b_{y,\eta}(t):= b\cdot u(y+ta)$ and with the hyper-plane orthogonal to $a$: $\pi_a:=\{\zeta\in\R^n:a\cdot\zeta=0\}$. Moreover it holds that $u^b_{y,a}\in BV(\Omega_y^a;\R)$ for $\H^{n-1}$-almost all $y\in\pi_a$.
 These results in $\BV$ and $\BD$ can be found in \cite{ ambrosio1997fine, ambrosio2000functions,rindler2018calculus}.
 
\subsubsection{Annihilators and structure property of the polar of $\mu=\AA u$}\label{sbs:ann}

Given the Radon measure $\sigma\in\mathcal M(\Om;W)$, let the Radon measure $\mu\in\mathcal M(\Om; V)$ satisfy
  \begin{align*}
   \mathcal B\mu= \sum_{j=1}^n B_j \partial_j \mu=\sigma\quad\mbox{ in } \mathcal D'(\Om;W),
  \end{align*}
with tensor coefficients $B_j\in\mathrm{Lin}(V,W)\sim W\times V^*$, with its \textit{symbol} $\mathbb B[\xi]:  V \rightarrow W$
\begin{align}
   \mathbb B[\xi] :=\sum_{j=1}^n \xi_j B_j .\nonumber
  \end{align}
The wave cone of $\mathcal B$ is defined as
\begin{align}
 \Lambda_{\mathcal{B}}:=\displaystyle\bigcup_{\xi\in S^{n-1}}\mathrm{Ker}   \mathbb B[\xi] \subset V. \label{wavecone}
  \end{align}
Under these conditions De Philippis and Rindler in \cite{de2016structure} have proved that the polar of $\mu$ has the following particular structure 
  \begin{align}
\frac{d\mu}{|d\mu|}\in  \Lambda_\mathcal B\quad |\mu|^s-\mbox{a.e. in } \Omega. 
  \end{align}
In $\BV(\Om;\R^m)$ and for $\mathcal B$ the \textit{curl operator}, one recovers Alberti's rank-one theorem \cite{alberti1993rank},
\begin{equation}\label{eqn:rankOneBV}
\frac{d D^su}{d |D^su|}(x)=a(x)\otimes b(x),
\end{equation}

for $D^su$-almost all $x\in\Om$ with $a(x),b(x)\in \R^n$. 
  In $\BD(\Om;\R^{n})$ and for $\mathcal B$ the \textit{incompatibility operator} \cite{AVG2016}, one finds
\begin{equation}\label{eqn:rankOneBD}
\frac{d \mathcal E^su}{d |\mathcal E^su|}(x)=a(x)\odot b(x),
\end{equation}
for $\mathcal E^su$-almost all $x\in\Om$, and
with $a(x),b(x)\in \R^n$.

\section{Iterated strict blow-ups}\label{sct:it}
We consider $\mathcal{A}$  to be a generic $\mathbb{C}$-elliptic operator and we consider $\BV^{\mathcal{A}}(\Omega;\R^m)$ to be the space of all the functions with bounded $\mathcal{A}$ variation from an open bounded regular set $\Omega$ into $\R^m$. For any fixed convex set $K$, $u\in \BV^{\mathcal{A}}(\Omega;\R^m)$, $x\in \Omega$ and $\e<\dist(x,\partial\Om)$ we define
	\[
	u_{K,\e,x}(y):=\frac{u(x+\e y) -\mathcal{R}_K[u(x+\e \cdot)](y)}{\frac{|\mathcal{A} u|(K_{\e}(x) )}{|K|\e^{n-1}}}.
	\]
being $\mathcal{R}_K: L^1(K;\R^m) \rightarrow  \mathrm{Ker}(\mathcal{A})$ any linear and bounded operator fixing $ \mathrm{Ker}(\mathcal{A})\cap L^1(K;\R^m)$. Note that $\mathcal{R}_K[u_{K,\e,x}]=0$, by the very Definition of the blow-up sequence.\\

We use standard notations for the push-forward of measures, and in particular, given 
$\mu\in \mathcal{M}_{loc}(\R^n;V)$, $x\in\R^n$ and $\e>0$, we will consider the push forward with the map 
$F^{x,\e}(y):=\frac{y-x}{\e}$ defined as
	\begin{equation}\label{e:pushforward}
	F_{\#}^{x,\e}\mu(A):=\mu(x+\e A).
	\end{equation}
Preiss' tangent space $\Tan(\mu,x)$ at a given point $x\in\R^n$, is defined as the subset of non zero measures 
 $\nu\in\mathcal{M}_{loc}(\R^n;V)$ 
such that $\nu$ is the local weak* limit 
of $\sfrac1{c_i}F_\#^{x,\e_i}\mu$, for some sequence $\e_i\downarrow 0$ 
as $i\uparrow+\infty$ and for some positive sequence $\{c_i\}_{i\in\N}$ 
(see  \cite{ambrosio2000functions}, \cite{mattila1999geometry}, \cite{rindler2018calculus}). 
\begin{remark}[Derivative of blow-ups]
Note that, with the given notations, an easy computation shows that
    \begin{equation}
        \mathcal{A}u_{K,\e,x}=|K| \frac{F^{x,\e}_{\#}\mathcal{A}u}{F^{x,\e}_{\#}|\mathcal{A}u|(K)}.\nonumber
    \end{equation}
Morever, any $L^1$-limit point $v$ of $\{u_{K,\e,x}\}_{\e>0}$ satisfies $ \mathcal{A}u_{K,\e,x}\wt \mathcal{A}v$, and hence has constant polar on $K$, i.e. satisfies
     \begin{equation}
    \label{eqn:derBU}
     \mathcal{A}v=\left(\frac{\d \mathcal{A}u } {\d  |\mathcal{A}u| }(x)\right)  |\mathcal{A}v| \ \ \ \text{on $K$}.
    \end{equation}
\end{remark}

To ensure that the total variation is preserved 
along the blow-up limit procedure we recall the ensuing result. 

\begin{lemma}[Tangent measure with unit mass, Lemma~10.6 \cite{rindler2018calculus}]\label{lem:unit mass conve} 
Let $\mu\in \mathcal{M}_{loc}(\R^n;V)$. 
Then, for $|\mu|$-a.e. $x\in \R^n$ and for every 
bounded, open, convex set $K$ the following assertions hold
\begin{itemize}
	\item[(a)] There exists a tangent measure $\gamma\in \Tan(\mu,x)$ such that $|\gamma|(K)=1$, 
	$|\gamma|(\partial K)=0$;
	\item[(b)] There exists 
	$ \e_i \downarrow 0$ as $i\uparrow+\infty$ such that 
	$\frac{F^{x,\varepsilon_i}_{\#}\mu}{F^{x,\e_i}_{\#}|\mu|(K)}\wt \gamma$ in $\mathcal{M}(\overline{K};V)$.
\end{itemize}	 
\end{lemma}
Up to the light of the previous Proposition it is convenient to introduce the following notion limit point for blow-ups (subordinated to the proof of Proposition \ref{lem:BB}).
\begin{definition}[Strict blow-up limit of $u$]\label{def:Strict}
Let $K$ be a convex set. We say that $(v,\gamma)\in \BV^{\mathcal{A}}(K;\R^m)\times \mathcal{M}(\ov{K};V)$ is a \textit{strict blow up limit} for $u\in \BV^{\mathcal{A}}(\Omega;\R^n)$ at a point $x$, and with respect to the convex set $K$ (and we write $(v,\gamma)\in \mathrm{bu}_{K}(u;x)$), if there is a vanishing sequence $\{\e_{i}\}_{i\in \N}$ such that
	\begin{itemize}
	\item[1)] $u_{K,\e_i,x}$ converges strictly to $v$ in $\BV^{\mathcal{A}}(K;\R^m)$:
 \[
 u_{K,\e_i,x}\rightarrow v\   \text{in $L^1$ and}\  |\AA u_{K,\e_i,x}|(K) \rightarrow  |\AA v|(K);
 \]
\item[2)] $\gamma \in \mathrm{Tan}(\mathcal{A}u,x)$ is such that 
\begin{align*}    
&\frac{1}{\L^n(K)} \mathcal{A}u_{K,\e_i,x} \wt \gamma \ \text{in}\  \mathcal{M}(\bar{K};V  )\\
&|\gamma|(K)=|\gamma (K)|=1, \ \ |\gamma|(\partial K)=0,\\
&\gamma =\frac{\d \mathcal{A}u}{\d |\mathcal{A}u|}(x) |\gamma| \ \ \text{for $|\gamma|$ -a.e. $x$ in $K$} \\
&\mathcal{A} v=\L^n(K) \gamma \llcorner_{K}.
\end{align*}
  
	\end{itemize}
\end{definition}
The following ensures that $ \mathrm{bu}_{K}(u;x)$ is never empty.
	\begin{lemma}\label{lem:BUnotEmpty}
	Let $K$ be a center symmetric convex set. Let $u\in \BV^\mathcal{A}(\Omega,\R^m)$. Then, for $|\mathcal{A}u|$-a.e. $x\in \Omega$ the set $\mathrm{bu}_K(u;x)$ is not empty. 
	\end{lemma}
	\begin{proof}
We apply Lemma \ref{lem:unit mass conve} to find $\gamma\in \mathrm{Tan}(\mathcal{A}u,x)$ and $\e_i \rightarrow 0$ such that
    \[
    \frac{F_{\#}^{x,\e_i} \mathcal{A} u}{F_{\#}^{x,\e_i} |\mathcal{A} u|(K)}\wt \gamma \ \ \ \text{in $\mathcal{M}(\overline{K};\R^m)$}.
    \]
Set $u_i:=u_{K,\e_i,x}$. Since $\mathcal{R}_K[u_{K,\e_i,x}]=0$, by \eqref{Poincare}
    \begin{align*}
        \|u_i\|_{L^{1}(K;V)}=  \|u_i-\mathcal{R}_K[u_i]\|_{{L^{1}(K;V)}}\leq c |\mathcal{A}u_i|(K)=c\L^n(K)
    \end{align*}
for some $c$ depending on $K,\mathcal{R}_K$ and $n$ only. This implies (see Theorem \ref{thm:cmp}) that, up to a further subsequence, $u_{i}\rightarrow v$ in $L^1(K;\R^m)$ with  $v\in \BV^{\mathcal{A}}(K;\R^m)$. In particular
    \[
    \mathcal{A}u_{i}\wt \mathcal{A} v.
    \]
Now by the properties of $\gamma$ and $v$ follows that $(v,\gamma)\in \mathrm{bu}_K(u;x)$.
	\end{proof}
Finally, with the help of the next result, we will be able to implement an iterated blow-up Lemma.
\begin{theorem}[Tangent measures to tangent measures are tangent measures, Theorem~14.16 \cite{mattila1999geometry}]\label{thm:Mattila}
Let $\mu\in \mathcal{M}_{loc}(\R^n;V)$ be a Radon measure. Then for $|\mu|$-a.e. $x\in \R^n$ any $\nu\in\Tan(\mu,x)$ satisfies the following properties
\begin{itemize}
\item[(a)] For any convex set $K$, $\frac{F^{y,\rho}_{\#} \nu}{F^{y,\rho}_{\#}|\nu|(K)} \in  \Tan(\mu,x)$  for all $y\in\spt \nu$ and $\rho>0$;
\item[(b)] $\Tan(\nu,y)\subseteq  \Tan(\mu,x)$ for all $y\in\spt\nu$;
\end{itemize}
\end{theorem}
Note that the original result in \cite[Theorem~14.16]{mattila1999geometry} is proven for $k=1$. However, the good properties of tangent space of measures (i.e. \cite[Theorem~2.44]{ambrosio2000functions} or
\cite[Lemma~10.4]{rindler2018calculus}) allow us to immediately extend its validity for generic $k$. 
\smallskip

We now prove the central result, repeatedly used in our argument, that will allow us to take advantage of the notion of \text{strict blow-up} limit. 


\begin{proposition}
[Blow-ups of blow-ups are blow-ups]\label{lem:BB}
Let $\mathcal{A}$ be a $\mathbb{C}$-elliptic operator. Fix $K_1, K_2$ two open bounded convex sets and pick $u\in \BV^{\mathcal{A}}(\Omega;\R^m)$. For $|\mathcal{A}u|$-a.e. $x\in \Omega$, any $h\in\BV^{\mathcal{A}}(K_1;\R^m)$ that is an $L^1$-limit point of $\{u_{K_1,\e,x}\}_{\e>0}$ satisfies  the following property:
\begin{itemize}
	\item[•] for $|\AA h|-$a.e. $y\in K_1$ if $(g,\gamma_g) \in \mathrm{bu}_{K_2}(h;y)$ then $(g-\mathcal{R}_{K_2}[g],\gamma_g)\in \mathrm{bu}_{K_2}(u;x)$.
\end{itemize}
\end{proposition}
\begin{remark}
Let us remark the role of $K_1$ and $K_2$. In principle $h$ is an $L^1$ blow-up of $u$, $\{u_{K_1,\e,x}\}_{\e>0}$ along the convex $K_1$ and $g$ is a strict blow-up of $h$ along the convex set $K_2$, $\{h_{K_2,\delta,y}\}_{\delta>0}$. The above Proposition in particular says that not only $g$ is a blow-up of $u$ but can also be reached as a sequence along $K_2$: $\{u_{K_2,\e_i,x}\}_{i\in \N}$. In particular we have the freedom to choose the convex set in doing the linearization procedure.
\end{remark}
\begin{proof}
We pick an $x\in \Omega$ for which  Theorem \ref{thm:Mattila} is in force with respect to the measure $\mathcal{A}u$ (which is $|\mathcal{A}u|$ almost everywhere in $\Omega$). Consider $h\in \BV^{\mathcal{A}}(K_1;\R^m)$ an $L^1$-limit point of $\{u_{K_1,\e,x}\}_{\e>0}$ and note that we can find a vanishing sequence $\{\e_i\}_{i\in \N}$ such that (in $L^1$)
	\[
	u_{K_1,\e_i,x}\rightarrow h.
	\] 
Notice that this means
	\[
	\mathcal{A}u_{K_1,\e_i,x}\wt \mathcal{A}h \ \ \ \text{in $\mathcal{M}(K_1;V)$}
	\]
and it also means
	\[
\L^n(K_1) \frac{F^{x,\e_i}_{\# } \mathcal{A}u}{F^{x,\e_i}_{\# } |\mathcal{A}u|(K_1)}=\mathcal{A}u_{K_1,\e_i,x} \wt \mathcal{A}h.
	\]
But then we have $\mathcal{A}h\in \mathrm{Tan}(\mathcal{A}u,x)$. In particular, because of Theorem \ref{thm:Mattila}, for all $y\in \spt(\mathcal{A}h)$, we can infer $\mathrm{Tan}(\mathcal{A}h,y)\subset \mathrm{Tan}(\mathcal{A}u,x)$. Fix one of such $y$ for which also Lemma \ref{lem:BUnotEmpty} holds (which are $|\AA h|$-a.e. $y\in K_1$) and notice that, for $(g,\gamma_g)\in \mathrm{bu}_{K_2}(h,y)$ we  can find a vanishing sequence $\{\delta_i\}_{i\in \N}$ such that
	\[
	h_{K_2,\delta_i,y}\rightarrow g
	\]
strictly in $\BV^{\mathcal{A}}(K_2;\R^n)$ and
	\[
	\frac{1}{\L^n(K_2)}\mathcal{A}h_{K_2,\delta_i,y}\wt \gamma_g \ \ \ \text{in $\mathcal{M}(\bar{K}_2;V)$}
	\]
and thus
	\[
\L^n(K_2) \frac{F^{y,\delta_i}_{\# } \mathcal{A}h}{F^{y,\delta_i}_{\# } |\mathcal{A}h|(K_2)}=\mathcal{A}h_{K_2,\delta_i,y} \wt \L^n(K_2) \gamma_g \ \ \  \text{in $\mathcal{M}(\bar{K}_2;V)$}
	\]
implying $\L^n(K_2)\gamma_g \in \mathrm{Tan}(\mathcal{A}h,y)$. We now know that the choice of $y$ ensures $\L^n(K_2)\gamma_g \in \mathrm{Tan}(\mathcal{A}h,y)\subset \mathrm{Tan}(\mathcal{A}u,x)$ and therefore it must exists a sequence $\{c_i\}_{i\in \N}$ and a vanishing sequence $\{\rho_i\}_{i\in \N}$ such that
	\[
	c_i F^{x,\rho_i}_{\# } \mathcal{A}u \wt \L^n(K_2)\gamma_g.
	 \]
Since $|\gamma_g|(\partial K_2)=0$ we have
	\[
	c_i F^{x,\rho_i}_{\# } |\mathcal{A}u|(K_2)\rightarrow \L^n(K_2)|\gamma_g|(K_2)=\L^n(K_2)
	\]
implying that
	\begin{equation}\label{gamm}
	\L^n(K_2) \frac{ F^{x,\rho_i}_{\# } \mathcal{A}u }{ F^{x,\rho_i}_{\# } |\mathcal{A}u |(K_2)}\wt \L^n(K_2) \gamma_g \ \ \ \text{in $\mathcal{M}(\bar{K_2};V)$}.
	\end{equation}
Since $\mathcal{A}g=\L^n(K_2) \gamma_g$ on $K_2$ we obtain  
	\[
 	\L^n(K_2)   \frac{ F^{x,\rho_i}_{\# } \mathcal{A}u }{ F^{x,\rho_i}_{\# } |\mathcal{A}u |(K_2)}\wt \mathcal{A}g\ \ \  \text{on $\mathcal{M}(K_2;V)$}.
	\]
Consider now $u_{K_2,\rho_i,x}$ and notice that 
\[
\| u_{K_2,\rho_i,x} \|_{L^{1}(K_2;\R^m)}\leq  c |\mathcal{A}u_{K_2,\rho_i,x}|(K_2)=c \L^n(K_2)
\]
and thus, up to extract a subsequence (not relabeled) we can infer $u_{K_2,\rho_i,x}\rightarrow \bar{g}\in \BV^{\mathcal{A}}(K_2;\R^m)$. In particular we also can deduce by this convergence that
	\[
	\mathcal{A}g=\mathcal{A}\bar{g} \ \ \ \Rightarrow \ \ \ g=\bar{g}+L \ \ \text{for some $L\in \mathrm{Ker}(\mathcal{A})$}.
	\]
This, combined with \eqref{gamm} implies that $(g-L,\gamma_g)=(\bar{g},\gamma_g)\in \mathrm{bu}_{K_2}(u;x)$. To identify $L$ we just use the fact that $\mathcal{R}_{K_2}$ is linear, continuous, keeps $\mathrm{Ker}(\mathcal{A})$ fixed and  $\mathcal{R}_{K_2}[u_{K_2,\rho_i,x}]=0$ for all $i\in \N$ yielding $	\mathcal{R}_{K_2}[g-L]= \mathcal{R}_{K_2}[\bar{g}]=0$. Thus
	\[
	\mathcal{R}_{K_2}[g]=L.
	\]
\end{proof}

\section{Application}\label{sct:app}

For a function $u$ differentiable at $x$ the following holds
$$
\frac{u(x+\e y)-u(x)}{\e |\nabla u(x) |}\stackrel{\e\to0}{\longrightarrow}\frac{\nabla u(x)}{ |\nabla u(x) |}y.
$$
This relation also holds at points of approximate differentiability of a $\BV$ function, where the gradient is understood as the approximate gradient \cite{ambrosio2000functions}. In this case, since $\frac{|D u|(K_{\e}(x) )}{|K_\e(x)|}\stackrel{\e\to0}{\longrightarrow} |\nabla u(x) |$, we can write the equivalent
relation
$$
u_{K,\e,x}(y)=\frac{u(x+\e y) - \left(u\right)_K}{\frac{|D u|(K_{\e}(x) )}{|K|\e^{n-1}}}\stackrel{\e\to0}{\longrightarrow}\frac{dD^au}{d|D^au|}(x)y=\frac{\nabla u(x)}{|\nabla u(x)|} y,
$$
where $\left(\cdot\right)_K$ denotes the average over the convex set $K$. 
The same holds $\L^d$-a.e. for function with bounded deformation (cf. \cite{ambrosio1997fine}).\\

In general, for a function with bounded $\AA$-variation, approximate differentiability implies (cf. for instance \cite{gmeineder2017critical}) that
$$
u_{K,\e,x}(y)=\frac{u(x+\e y) - \mathcal R_K[u](y)}{\frac{|\AA u|(K_{\e}(x) )}{|K|\e^{n-1}}}\stackrel{\e\to0}{\longrightarrow}\frac{d\AA^a  u}{d|\AA^a u|}(x)y,
$$
at $\L^{n}$-a.e. point. Additionally, using the definition of \textit{jump points}, it is quite a standard argument to verify that, for $\H^{n-1}$-a.e. $x\in J_u$, we have $u_{K,\e,x}(y)\stackrel{\e\to0}{\longrightarrow}H_{u^+(x),u^-(x),\nu(x)}(y)$ where
\begin{equation}
   H_{u^+,u^-,\nu}(y):=\left\{
   \begin{array}{ll}
u^+ \  &\text{if $y\cdot \nu>0$}\\
u^- \ &\text{if $y\cdot \nu<0$}.
    \end{array}\right.\nonumber
\end{equation}
By means of iterative blow-ups we will show that also at Cantor points a sequence of $\e$ can be selected to converge to an affine function capturing the local behavior of the polar vector of the gradient measure, in the case $\AA=D$ ($\BV$-functions) and $\AA=\mathcal E$ ($BD$-functions). Notice that, unlike the case of approximate differentiable points and jump points, where the blow-ups converge to a unique limit, at Cantor points we can only select a good sequence along which the blow-ups converge to a good limit. This is due to the highly irregular structure that the blow-ups can have around Cantor-type points. However, in typical homogenization and relaxation results, having one good blow-up is already enough for the applications.\\

The iterated blow-ups strategy can be described as a repeated application of Proposition \ref{lem:BB}, a rigidity structure results (as the one stated in Proposition \ref{prop:rindlerchar}), and the structure of the operator $\mathcal{R}$ appearing in the Poincaré inequality. For instance, considering the case $\AA=\E$ the general strategy (following the milestones in \cite{caroccia2019integral} and \cite{de2019fine}) can be illustrated like this:
 \begin{itemize}
     \item[1)] If $x$ is a point of approximate differentiability or a jump point, then blow-ups are given by the approximate differential $e(u)(x)y$ or by the jump function $H$;
     \item[2)] If $x$ is neither a jump point nor an approximate differentiability point then, by \cite{de2016structure}, we have that $    \frac{\d \E^s u}{\d |\E^s u|}$ must have a precise shape (i.e., it belongs to the wave cone of the annihilator of $\E$);
    \begin{itemize}
        \item[2.1)] Then we consider a blow-up $v$, $L^1$-limit of $u_{K,\e_i,x}$. This blow-up will have the property of satisfying
     \[
     \E v= \frac{\eta \odot \xi }{|\eta \odot \xi| }|\E  v| \ \ \text{on $K$};
     \]
     \item[2.2)] The rigidity result in Proposition \ref{prop:rindlerchar} will now yield information on the shape of $v$. In particular (in the $\BD$ case) we have 
     \[
     v(x)=\eta h_1(x\cdot \xi)+\xi h_2(x\cdot \eta)+L(x);
     \]  
     \item[2.3)] If $x\notin \mathrm{TS}(u)$, by selecting a specific point on the domain we perform a second blow-up on $v$, which linearizes both directions, resulting in one blow up of the form 
     \[
     w(y)=(\kappa_1\xi\otimes \eta +  \kappa_2 \eta \otimes \xi)y;
     \]
     \item[2.4)] By employing the information on $\mathcal{R}$ (with a further application of the rigidity result if $\eta$ and $\xi$ are parallel) the constant $\kappa_i$ is proven to be $\kappa_i  =\frac{1}{2|\eta\odot \xi|}$ yielding
     \[
     g(y)=\frac{\eta \odot \xi}{|\eta\odot \xi|} y=\frac{\d \E u}{\d |\E u|}(x) y ;
     \]
     \item[2.5)] By employing Proposition \ref{lem:BB} we conclude that $g$ can be obtained as the strict limit of  $u_{K,\tilde{\e}_i,x} $ for some $\tilde{\e} _i \downarrow 0$.
    \end{itemize}
 \end{itemize}
In particular we show that (excluding the small family of \textit{totally singular} points that we identify in step 2.3) almost all Cantor points must have at least a blow-up which is $\frac{\d \E u}{\d |\E u|}(x) y $. If a points $x\in \mathrm{TS}(u)\setminus J_u$ - still by applying Proposition \ref{lem:BB} and given the rigidity result - we can always provide a blow-up affine at least in one direction as done in \cite{caroccia2019integral}. With these information at hand, it is now easy to proceed to relaxation and integral representation. Note that the final affine shape is not a surprise since the blow-ups are defined by removing the lower order terms $\mathcal{R}[u(x+\e \cdot)]$ (and this is crucial in giving the limiting affine shape; otherwise some contribution from $\mathcal{R}[u(x+\e \cdot)]$ might arise, requiring much more knowledge of the pointwise behavior of $u$ around $x$). \\

The excluded set $\mathrm{TS}(u)$ is made by those point $x$ for which all the $L^1$ blow-ups $h$ are singular: $\E h=\E^s h$. This is a very specific property and brings to a natural question: how does a function $u$ look if all its blow-ups have zero absolutely continuous part and can be expressed as sums of one-dimensional functions?\\
 
Generally speaking, the above strategy works also for a generic $\CC$-elliptic operator $\AA$. Indeed the crucial ingredients are: a rigidity result describing the structure of functions with constant polar $\AA u = P |\AA u|$, and a specific choice of the application $\mathcal{R}$ appearing in the Poincaré inequality \ref{Poincare}. While the choice of $\mathcal{R}$ can be done in rather general contexts, the rigidity structure seems, instead, specific to each $\AA$ and that would be where the main difficulties we believe to rely. Nonetheless, we see in this strategy a fruitful starting point for pursuing homogenization and integral representation results at least for some specific operators, such as the deviatoric strain $\E_{\mathrm{dev}} u:= \E u - \frac{\dive(u)}{d}\Id$, whose rigidity result will be published in a subsequent work by the authors.

\subsection{Affine blow-ups in $\BV$}
We conclude by showing a similar statement in the context of $\BV$ maps.
 \begin{definition}[Totally singular points for $\BV$ maps]\label{def:TSpoinBV}
For $u\in \BV(\Omega;\R^n)$, consider the blow-up sequences
\[
	u_{K,\e,x}(y):=\frac{u(x+\e  y)-(u)_{K_{\e}(x)}}{\frac{|Du|(K_{\e  }(x))}{|K|\e^{n-1}}}.
	\]
A point $x\in\Om$ is said $\BV$-totally singular for $u\in \BV(\Om;\R^m)$ if for any $L^1$ limit points $p\in \BV(K;\R^m)$ of $\{u_{K,x,\e} \}_{\e>0}$ it holds $D p=D^s p$. We denote by $\mathrm{TS}(u)$ the set of points which are $\BV$-totally singular for $u$.
 \end{definition}
 
\begin{theorem}\label{thm:mainBV}
Let $n\geq 2$ and $u\in \BV(\Omega;\R^m)$. Then, for any convex set with $K=-K$ and for $|D^c u|$-a.e. $x\in \Omega\setminus \mathrm{TS}(u)$ there exists a vanishing sequence $ \e_i \downarrow 0 $ such that
 	\[
	u_{K,\e_i,x}(y)\rightarrow  \frac{\d D u}{\d |D u| } (x) y 
	\]
strictly on $\BV(K;\R^m)$.
\end{theorem} 

The proof of Theorem \ref{thm:mainBV} it is an easy application of Proposition \ref{lem:BB}.
\begin{proof}[Proof of Theorem \ref{thm:mainBV}]
Having chosen a Cantor point $x$ such that
	\[
	\frac{\d Du }{\d |Du|}(x)=\frac{\eta\otimes \xi}{|\eta\otimes \xi|},
	\]
(which are $|D^c u|$ almost all in $\Omega$) thanks to the rigidity of $\BV$ blow-ups (see for instance \cite{ambrosio2000functions}) we can infer that any $\BV$ limit point $p$ of $\{u_{K,\e,x}\}_{\e>0}$ has the shape
	\[
	p(y)=\psi(y\cdot \eta)\xi+c
	\] 
for some $c,\xi\in \R^m$, $\eta\in \R^n$ and $\psi\in \BV_{loc}(\R)$. Since $x\notin \mathrm{TS}(u)$ we can find at least one $\BV$ limit point with a $\psi$ such that $\psi'\neq 0$ on a set of positive Lebesgue measure. Then, on setting $\Psi(z):=\psi(z\cdot \eta)\xi$ we have $\Psi'\neq 0$ on a set with positive Lebesgue measure. We  thus can select a point $y\in K$ such that
 	\begin{enumerate}
 	\item[a)] $y \in  \spt(Dp)$ with $y\cdot \eta$ a point of approximate differentiability for $\psi$ and a Lebesgue point for $y\mapsto \frac{\d D p}{\d|D p|}(y)$;
 			\item[b)] $\frac{\Psi(y+\rho z ) - (\Psi)_{ K_{\rho}(y)} }{\rho}\rightarrow \beta_1 (\eta \otimes \xi )z+c$ in $L^1(K)$ for some $\beta_1\in \R$, $c\in \R^m$;
 	\item[c)] $\frac{|Dp|(K_{\rho}(y))}{\L^n(K_{\rho}(y))}\rightarrow  \beta_3>0$.
 	\end{enumerate}
We now pick $(g,\gamma_g)\in \mathrm{bu}_K(p;y)$ (which is not empty thanks to Lemma \ref{lem:BUnotEmpty}). Notice that by applying Proposition \ref{lem:BB} we have $(g-(g)_{K},\gamma)\in \mathrm{bu}_K(u;x)$. Recall that for $\BV$: 
\[
\mathcal{R}_K(w)=(w)_K:=\fint_K w\d x.
\]
In particular $g-(g)_{K}$ can be reached as a $\BV$-strict limit point of $\{u_{K,\e,x}\}_{\e>0}$. It is now enough to characterize $g-(g)_{K}$ as a blow-up of $p$.\\

\textit{Characterization of the blow-up $g$:} notice that
	\begin{align*}
	p_{K,\rho_i,y}(z)=\frac{\Psi(y+\rho_i z)   - (\Psi (y+\rho_i\cdot))_{K}}{\frac{|D  p|(K_{\rho_i}(y))}{|K|\rho_i^{n-1}}}.
	\end{align*}
Since 
\[
(\Psi (y+\rho_i\cdot))_{K}=(\Psi)_{K_{\varrho_i}(y)}
\]
we just conclude that
\begin{align*}
	p_{K,\rho_i,y}(z)=\frac{\Psi(y+\rho_i z)   - (\Psi (y+\rho_i\cdot))_{K}}{\frac{|D  p|(K_{\rho_i}(y))}{|K|\rho_i^{n-1}}}=\frac{\Psi(y+\rho_i z)   - (\Psi )_{K_{\varrho_i}(y)}}{\varrho_i\frac{|D  p|(K_{\rho_i}(y))}{|K|\rho_i^{n}}}\rightarrow \frac{\beta_1}{\beta_3}(\eta \otimes \xi) z + c 
	\end{align*}
in $L^1(K;\R^m)$. Since $g(z)-(g)_{K}\in \mathrm{bu}_K(u;x)$ and $(g)_K=c$ we get 
\[
g(z)-(g)_K=\frac{\beta_1}{\beta_3}(\eta \otimes \xi) z.
\]
Also we have strict convergence of the blow-ups and this implies that $|D g|(K)=|D p_{K,\varrho_i,x}|(K)=\L^n(K)$ implying
\[
\frac{\beta_1}{\beta_3}=\frac{1}{|\eta\otimes \xi|}
\]
and thus
\[
g(z)-(g)_{K}=\left(\frac{\eta\otimes \xi}{|\eta\otimes \xi|}\right)z=\frac{\d Du}{\d|Du|}(x) z \in \mathrm{bu}_K(u;x).
 	\]
\end{proof}

\subsection{Affine blow-ups in $\BD$}

Beyond the structure Theorems given in Subsection \ref{sbs:structure}, we can refer - for $\E$ - to the fine properties obtained in \cite{ambrosio1997fine}. For $u\in \BD(\Omega)$ the following well- known splitting is in force
\[
\E u= e(u)\L^n + [u]\odot \nu_{u} \H^{n-1}\llcorner_{J_u}+\E^c u
\]
where $e(u)(x):=\frac{d\E u}{\d \L^n}(x)$ (which can be computed as $e(u)(x)=\frac{\nabla u(x)+\nabla u(x)^t}{2}$, being $\nabla u(x)$ the approximate differential of $u$ at $x$, existing $\L^n$-a.e. on $\Omega$ - cf. \eqref{eqn:appdiff}),  $J_u$ the \textit{jump set of $u$}, $\nu_u$ a unitary vector field normal to $J_u$, $[u]:=u^+-u^-$ the jump set, $a\odot b:=\frac{a\otimes b + b\otimes a}{2}$ and $\E^c u$ stands for the cantor part of the measure $\E u$, supported on $C_u$:
\begin{equation}\label{e:Cu}
C_u:=\Big\{x\in\Omega\setminus S_u :\,\textstyle{\lim_{r\downarrow 0}\frac{|Eu|(B_r(x))}{r^n}=+\infty,\,
\lim_{r\downarrow 0}\frac{|Eu|(B_r(x))}{r^{n-1}}=0}\Big\}.
\end{equation}
Recall also that (\cite[Theorem 6.1]{ambrosio1997fine}) it holds 
\[
|\E u|(S_u\setminus J_u)=0.
\]

As a consequence of the more general result in \cite{de2016structure}, described in Subsection \ref{sbs:ann} for $|\E^c u|-$a.e. $x\in \Omega$ we have
\[
\frac{\d \E^c u}{\d| \E^c u|}(x)= \frac{\eta(x)\odot \xi (x) }{|\eta(x)\odot \xi (x) |} 
\]
for some Borel $|\E^c u|$-measurable vector fields $\eta,\xi:\Omega\rightarrow \R^n$.\\

We report a rigidity result for $\BD$ maps with constant polar vector. The result can be found in \cite[Theorem~2.10 (i)-(ii)]{de2019fine} (see also \cite[Proposition 3.9]{caroccia2019integral}).

 \begin{proposition}[Rigidity] \label{prop:rindlerchar}
 If $w\in \BD_{loc}(\R^n)$ is such that for some $\eta,\,\xi\in\R^n$
	\begin{equation}\label{e:cantordec}
	\E w= \frac{\eta\odot \xi}{|\eta\odot \xi|}|\E w|,
	\end{equation}
then 
\begin{itemize}
 \item[(i)] if $\eta\neq\pm\xi$ 
 	\[
	w(y)=\alpha_1(y\cdot \xi)\eta+\alpha_2(y\cdot \eta)\xi+\mathrm{L} y+\mathrm{v},
	\]
for some $\alpha_1,\alpha_2\in \BV_{loc}(\R)$, $\mathrm{L}\in \mathbb{M}_{skew}^{n\times n}$, $\mathrm{v}\in \R^n$;

\item[(ii)] if $\eta=\pm\xi$
 	\[
	w(y)=\alpha(y\cdot \eta)\eta+\mathrm{L} y+\mathrm{v},
	\]
for some $\alpha\in \BV_{loc}(\R)$, $\mathrm{L}\in \mathbb{M}_{skew}^{n\times n}$, $\mathrm{v}\in \R^n$.
\end{itemize}
\end{proposition}

We then recall the following particular kernel projection (cf. \cite[Lemma 3.5, Proposition 3.6] {caroccia2019integral}, that is classical when $K=B$ is the ball (see \cite{ambrosio1997fine,kohn1980new}). We recall that 
\[
 \mathrm{Ker}(\mathcal{E}:=\{ z(y):= \mathrm{L}y+\mathrm{b} \ | \ \mathrm{L}\in \mathbb{M}_{skew}^{n\times n}, \ \mathbb{b}\in \R^n\},
\]

\begin{lemma}\label{l:auxiliary}
Let $K$ be a center-symmetric convex set (i.e. $K=-K$). For $u\in \BD(K)$ we define $\mathcal{R}_K[u]$ to be the affine map
\[
\mathcal{R}_K[u](y) :=\mathrm{L}_K[u]y+\mathrm{b}_K[u]
\]
with
\begin{align*}
\mathrm{L}_K[u]&:=\frac{1}{2\L^{n}(K)}\int_{\partial K} \big(u \otimes \nu_{\partial K} - \nu_{\partial K} \otimes u\big)\d\H^{n-1}\\
\mathrm{b}_K[u]&:=\frac{1}{\H^{n-1}(\partial K)}\int_{\partial K} u\d\H^{n-1}.
\end{align*}
Then $\mathcal{R}_K:\BD(K)\rightarrow\mathrm{Ker}(\mathcal{E})$ extends to a linear, bounded functional on $L^1(K;\R^n)$ with values in $\mathrm{Ker}(\mathcal{E})$ and $\mathcal{R}_K[p]=p$ for all $p\in \mathrm{Ker}(\mathcal{E})$.
\end{lemma}
In particular $\mathcal{R}$ can be used to define the blow-ups, and in the Poincaré inequality \eqref{Poincare}.\\
Note that for an affine function and for a center-symmetric convex set $K$, one has, by a simple integration by parts,
\begin{align}
\mathcal{R}_K[Ax+b](y)=\frac{\left(A-A^T\right)}{2}y+b\label{skew}
\end{align}

With these notion retrieved we can now prove the blow-up selection principle Theorem. We first introduce the Definition of \textit{totally singular points} that will make the statement more clean.
 
\begin{definition}[Totally singular points for $\BD$ maps]\label{def:TSpoinBD}
For $u\in \BD(\Omega)$, consider the blow-up sequences
\[
	u_{K,\e,x}(y):=\frac{u(x+\e  y)- \mathcal{R}_{K}[u(x+\e\cdot)](y)}{\frac{|\E u|(K_{\e  }(x))}{|K|\e^{n-1}}}.
	\]
A point $x\in\Om$ is said to be a \textit{totally singular point} for $u\in \BD(\Om)$ if  for any $L^1$-limit point $p\in \BD(K)$ of $\{u_{K,\e,x} \}_{\e>0}$ it holds $\E p=\E^s p$. We denote by $\mathrm{TS}(u)$ the set of points which are totally singular for $u$. 
 \end{definition}

\begin{theorem}\label{thm:blupSel}
Let $u\in\BD(\Omega)$. Let $K$ be a center-symmetric convex set. Then for $|\E^c u|$-a.e. $x\in\Omega\setminus  \mathrm{TS}  (u)$ there exists a sequence $\e_i \downarrow 0$ such that
\[
u_{K,\e_i,x}(y)\rightarrow \frac{\d \E u}{\d |\E u|}(x)y \ \ \text{strictly in $\BD(K)$ }
\]
\end{theorem}

\begin{proof}[Proof of Theorem \ref{thm:blupSel}]
Since $x$ is fixed for the rest of the proof, we simply denote by $\eta,\xi$ the vectors of the polar. We prove the statement at any $x\in \Omega\setminus\mathrm{TS}(u)$ such that Proposition  \ref{lem:BB} is in force (which are $|\E^cu|$-almost all $ x\in \Omega $). Since $x\notin \mathrm{TS}(u)$ then there exists an $L^1$ limit point $h$ of $\{u_{K,\e_i,x}\}_{i\in \N}$ such that $e(h)\L^n \neq 0$. Because of Proposition \ref{prop:rindlerchar} we have
 \[
 h(y)=\a_1(y\cdot \eta) \xi + \a_2(y\cdot \xi)\eta + \mathrm{L}y + \mathrm{v}
 \]
 for some $\a_i\in \BV_{\loc}(\R),$ $\mathrm{L}\in \mathbb{M}_{skew}^{n\times n}, \ v\in \R^n$. Since $e(u)\neq 0$ it follows that $\a_1'(y\cdot \eta) + \a_2'(y\cdot \xi)\neq 0$. 
Set $\Psi_1(y):= \alpha_1(y\cdot \eta) \xi, \Psi_2(y):=\a_2(y\cdot \xi)\eta $ and let us now select $y\in K$ satisfying 
 	\begin{enumerate}
 	\item[a)] $y \in  \spt(\E h)$ with $y\cdot \eta$, $y\cdot \xi$ a point of approximate differentiability (respectively) for $\a_1,\a_2$ and a Lebesgue point for $y\mapsto \frac{\d \E h}{\d|\E h|}(y)$;
 	\item[b)] $\frac{\Psi_1(y+\rho \cdot ) - (\Psi_1)_{\partial K_{\rho}(y)} }{\rho}\rightarrow   \beta_1^{(1)} \xi \otimes \eta+c_1$ in $L^1(K;\R^n)$ for some $\beta_1^{(1)}\in\R$, $c\in \R^n$; 
    \item[c)] $\frac{\Psi_2(y+\rho \cdot ) - (\Psi_2)_{\partial K_{\rho}(y)} }{\rho}\rightarrow   \beta_1^{(2)} \eta\otimes \xi +c_2$ in $L^1(K;\R^n)$ for some $\beta_1^{(2)}\in\R$, $c\in \R^n$;    
 	\item[d)] $\frac{D\Psi_1 (K_{\rho}(y))}{\L^n(K)  \rho^n}\rightarrow \beta_2^{(1)} \xi\otimes \eta$;
    \item[e)]$\frac{D\Psi_2 (K_{\rho}(y))}{\L^n(K)  \rho^n}\rightarrow \beta_2^{(2)} \eta\otimes \xi$;
 	\item[d)] $\frac{|\E h|(K_{\rho}(y))}{\L^n(K_{\rho}(y))}\rightarrow  \beta_3>0$.
 	\end{enumerate}
 Since $x\notin \mathrm{TS}(u)$ (and since $e(h)=(\Psi_1' + \Psi_2')\eta\odot  \xi\neq0$) we can guarantee that, for a set of positive $|\E h|$ measure we have $\beta_3 = \beta_2^{(1)} +\beta_2^{(2)}\neq  0 $. \\

Consider now $(g,\gamma)\in \mathrm{bu}_K(h;y)$ (which is not empty due to Lemma \ref{lem:BUnotEmpty}) and notice that, as in the proof of \ref{thm:mainBV}, by applying Proposition \ref{lem:BB} we have $(g-\mathcal{R}_K[g],\gamma)\in \mathrm{bu}_K(u;x)$. In particular $g-\mathcal{R}_K[g]$ can be reached as a $\BD$-strict limit point of $\{u_{K,\e,x}\}_{\e>0}$. It is now enough to characterize $g-\mathcal{R}_K[g]$ as a blow-up of $h$.\\

\textit{Characterization of the blow-up $g$:} Notice that
	\begin{align*}
	h_{K,\rho_i,y}(z)=\frac{\Psi_1(y+\rho_i z)  +  \Psi_2(y+\rho_i z)   - \mathcal{R}_K[h(y+\rho_i\cdot) ](z)}{\frac{|\E h|(K_{\rho_i}(y))}{|K|\rho_i^{n-1}}} 
	\end{align*}
	and that, by \eqref{skew},
	\begin{align*}
 \mathcal{R}_K[h(y+\rho_i\cdot) ](z)&= \frac{|K|\rho_i^{n-1}}{|\E h|(K_{\rho_i}(y))}\left( \mathcal{R}_K[\Psi_2(y+\rho_i \cdot)](z) +  \mathcal{R}_K\left[\Psi_1(y+\rho_i  \cdot)\right](z)\right)
	\end{align*}
	Moreover
	\begin{align*}
	\mathrm{L}_K[\alpha_1((y+\rho_i \cdot )\cdot \eta )\xi ]&=\frac{1}{2|K|}\int_{\partial K}\alpha_1((y+\rho_i z)\cdot \eta) (\xi \otimes \nu_K- \nu_K\otimes \xi)\d\H^{n-1}(z)\\
	&=\frac{1}{2|K|\rho_i^{n-1}}\int_{\partial K_{\rho_i}(y) }\alpha_1(z\cdot \eta) (\xi \otimes \nu_{K_{\rho_i}(y) }- \nu_{K_{\rho_i}(y) }\otimes \xi)\d\H^{n-1}(z)\\
	&=\frac{\rho_i}{2|K_{\rho_i}(y)|}\left[D\Psi_1 (K_{\rho_i}(y))- D\Psi_1 (K_{\rho_i}(y))^t\right]
	\end{align*}
    and analogously
   \begin{align*}
	\mathrm{L}_K[\alpha_2((y+\rho_i \cdot )\cdot \xi )\eta ]&=  \frac{\rho_i}{2|K_{\rho_i}(y)|}\left[D\Psi_2 (K_{\rho_i}(y))- D\Psi_2 (K_{\rho_i}(y))^t\right]
	\end{align*} 
	Thence 
	\begin{align*}
	h_{K,\rho_i,y}(z)=&\frac{|K| \rho_i^{n-1}}{|\E h| (K_{\rho_i}(y))}\left[ \Psi_1(y+\rho_i z)- (\Psi_1)_{\partial K_{\rho_i}(y)} - \frac{\rho_i}{2|K_{\rho_i}(y)|}\left[D\Psi_1 (K_{\rho_i}(y))- D\Psi_1 (K_{\rho_i}(y))^t\right] z\right.\\
 &\left. + \Psi_2(y+\rho_i z)- (\Psi_2)_{\partial K_{\rho_i}(y)} - \frac{\rho_i}{2|K_{\rho_i}(y)|}\left[D\Psi_2 (K_{\rho_i}(y))- D\Psi_2 (K_{\rho_i}(y))^t\right] z
 \right]
	\end{align*}
First we see that, thanks to our choice of $y$ we have 
	\begin{align*}
\frac{\frac{\rho_i}{2|K_{\rho_i}(y)|}\left[D\Psi_j (K_{\rho_i}(y))- D\Psi_j (K_{\rho_i}(y))^t\right]}{ \frac{|\E h|(K_{\rho_i}(y))}{|K|\rho_i^{n-1}} } &\rightarrow  \frac{\beta_2^{(j)}}{2\beta_3}\left(\xi\otimes \eta- \eta\otimes \xi \right) \ \ \ \ \text{(because of hypothesis d), e), f) )}.
	\end{align*}
Also that, because of hypothesis b), c)
	\[
	\frac{\Psi_j(y+\rho_i z) - (\Psi_j)_{\partial K_{\rho_i}(y)}}{\rho_i} \rightarrow \beta_1^{(j)}( \xi \otimes \eta ) z + c
	\]
in $L^1$ and thus
	\begin{align*}
	\frac{\Psi_j(y+\rho_i z) - (\Psi_j)_{\partial K_{\rho_i}(y)}}{\frac{|\E h|(K_{\rho_i}(y) ) }{|K|\rho_i^{n-1} }} \rightarrow \frac{\beta_1^{(j)}}{\beta_3}( \xi \otimes \eta )z + \frac{c_j}{\beta_3}
	\end{align*}
in $L^1(K;\R^n)$. Thence we conclude that
	\[
	h_{K,\rho_i,y}\rightarrow \kappa_1 (\eta\otimes \xi)z +\kappa_2( \xi\otimes \eta)z+C
	\]
in $L^1(K;\R^n)$ and for some $\kappa_1,\kappa_2\in \R$, $C\in \R^n$.  Since $h_{K,\rho_i,y}\rightarrow g$ strictly in $\BD(K)$ we conclude  
	\[
	g(z)= \kappa_1 (\eta\otimes \xi)z +\kappa_2( \xi\otimes \eta)z+C.
	\]
We now know that $g-\mathcal{R}_K[g]$ is also a $\BD$-strict limit point of $\{u_{K,\e,x}\}_{\e>0}$. In particular, setting $\mathcal{R}_K[g](z)=Lz+b$ for some $L\in\mathbb{M}^{n\times n}_{skew}$, by \eqref{skew}, we see that
	\[
 \frac{\kappa_1-\kappa_2}{2}\left(\eta\otimes \xi-\xi \otimes \eta\right)-L=0, \ \ \ C-b=0
	\]
which means $C=b$ and 
	\[
		L=\frac{\kappa_1-\kappa_2}{2}\left(\eta\otimes \xi-\xi \otimes \eta\right).
	\]
Thus
	\[
	g(z)-\mathcal{R}_K[g](z)=\kappa (\eta\odot \xi)z
	\]
with $ \kappa= \frac{\kappa_1+\kappa_2}{2}$. We now combine this information with the fact that $|\E g|(K)=\L^n(K)$, again  implied by the strict convergence,  to obtain
	\[
	\L^n(K)=|\E g|(K)=\kappa|\eta\odot \xi| \L^n(K) \ \ \Rightarrow \ \ \kappa=\frac{1}{|\eta\odot \xi|}
	\]
achieving 
	\[
g(z)-\mathcal{R}_K[g](z)=\left(\frac{\eta\odot \xi}{|\eta\odot \xi|}\right)z=\frac{\d \E u}{\d|\E u|}(x) z.
	\]
\end{proof}
\begin{remark}
Notice that $x\notin \mathrm{TS}(u)$ is crucial in the above argument since a further linearization through an additional blow-up can be performed only if we can find a $\BD$ limit point $h$ with $\E h \neq \E^s h$. 
\end{remark}


\textbf{Acknowledgments:} MC thanks the financial support of PRIN 2022R537CS "Nodal optimization, nonlinear elliptic equations, nonlocal geometric problems, with a focus on regularity" funded by the European Union under Next Generation EU.  NVG was supported by the FCT  project UIDB/04561/2020. The authors are deeply grateful to F. Gmeineder for the content of Proposition \ref{prop:sobBalls}.

\bibliography{references}

\begin{thebibliography}{10}

\bibitem{alberti1993rank}
G.~Alberti.
\newblock Rank one property for derivatives of functions with bounded
  variation.
\newblock {\em Proceedings of the Royal Society of Edinburgh Section A:
  Mathematics}, 123(2):239--274, 1993.

\bibitem{alberti2014p}
G.~Alberti, S.~Bianchini, and G.~Crippa.
\newblock On the $ {L}^p $-differentiability of certain classes of functions.
\newblock {\em Revista Matem{\'a}tica Iberoamericana}, 30(1):349--367, 2014.

\bibitem{ambrosio1997fine}
L.~Ambrosio, A.~Coscia, and G.~Dal~Maso.
\newblock Fine properties of functions with bounded deformation.
\newblock {\em Archive for Rational Mechanics and Analysis}, 139(3):201--238,
  1997.

\bibitem{ambrosio2000functions}
L.~Ambrosio, N.~Fusco, and D.~Pallara.
\newblock {\em Functions of bounded variation and free discontinuity problems},
  volume 254.
\newblock Clarendon Press Oxford, 2000.

\bibitem{AVG2016}
S.~{Amstutz} and N.~{Van Goethem}.
\newblock {Analysis of the incompatibility operator and application in
  intrinsic elasticity with dislocations}.
\newblock {\em {SIAM J. Math. Anal.}}, 48(1):320--348, 2016.

\bibitem{ap3}
O.~Anza~Hafsa and J.-Ph. Mandallena.
\newblock Stochastic homogenization of nonconvex integrals in the space of
  functions of bounded deformation.
\newblock {\em Asymptotic Analysis}, 131(2):209 – 232, 2023.

\bibitem{arroyo2020slicing}
A.~Arroyo-Rabasa.
\newblock Slicing and fine properties for functions with bounded a-variation.
\newblock {\em arXiv preprint arXiv:2009.13513}, 4, 2020.

\bibitem{arroyo2019dimensional}
A.~Arroyo-Rabasa, G.~De~Philippis, J.~Hirsch, and F.~Rindler.
\newblock Dimensional estimates and rectifiability for measures satisfying
  linear pde constraints.
\newblock {\em Geometric and Functional Analysis}, 29(3):639--658, 2019.

\bibitem{arroyo2019fine}
A.~Arroyo-Rabasa and A.~Skorobogatova.
\newblock On the fine properties of elliptic operators.
\newblock 2019.

\bibitem{ap2}
A.~C. Barroso, J.~Matias, and E.~Zappale.
\newblock Relaxation for an optimal design problem in $\bd(\omega)$.
\newblock {\em Proceedings of the Royal Society of Edinburgh Section A:
  Mathematics}, 153(3):721 – 763, 2023.

\bibitem{bouchitte1998global}
G.~Bouchitt{\'e}, I.~Fonseca, and L.~Mascarenhas.
\newblock A global method for relaxation.
\newblock {\em Arch. Rational Mech. Anal.}, 145(1):51--98, 1998.

\bibitem{breit2017traces}
D.~Breit, L.~Diening, and F.~Gmeineder.
\newblock {On the trace operator for functions of bounded
  {$\mathbb{A}$}-variation}.
\newblock {\em Analysis e PDE}, 13(2):559 -- 594, 2020.

\bibitem{Calderon1952}
A.~P. Calderon and A.~Zygmund.
\newblock {On the existence of certain singular integrals}.
\newblock {\em Acta Mathematica}, 88(none):85 -- 139, 1952.

\bibitem{caroccia2019integral}
M.~Caroccia, M.~Focardi, and N~Van Goethem.
\newblock On the integral representation of variational functionals on bd.
\newblock {\em SIAM Journal on Mathematical Analysis}, 52(4):4022--4067, 2020.

\bibitem{ap0}
M.~Cicalese, M.~Focardi, and C.~I. Zeppieri.
\newblock Phase-field approximation of functionals defined on piecewise-rigid
  maps.
\newblock {\em Journal of Nonlinear Science}, 31(5), 2021.

\bibitem{ap1}
V.~Crismale, M.~Friedrich, and F.~Solombrino.
\newblock Integral representation for energies in linear elasticity with
  surface discontinuities.
\newblock {\em Advances in Calculus of Variations}, 15(4):705 – 733, 2022.

\bibitem{de2016structure}
G.~De~Philippis and F.~Rindler.
\newblock On the structure of $\mathcal{A}$-free measures and applications.
\newblock {\em A.ls of Mathematics}, pages 1017--1039, 2016.

\bibitem{de2019fine}
G.~De~Philippis and F.~Rindler.
\newblock Fine properties of functions of bounded deformation-an approach via
  linear pdes.
\newblock {\em Mathematics in Engineering}, 2(3):386--422, 2020.

\bibitem{del2021rectifiability}
G.~Del~Nin.
\newblock Rectifiability of the jump set of locally integrable functions.
\newblock {\em A.li della Scuola Normale Superiore di Pisa, Classe di Scienze},
  22(3):1233--1240, 2021.

\bibitem{Diening1}
L.~Diening and F.~Gmeineder.
\newblock Continuity points via {R}iesz potentials for {$\mathbb C$}-elliptic
  operators.
\newblock {\em Quarterly Journal of Mathematics}, 71(4):1201 – 1218, 2020.

\bibitem{diening2021sharp}
L.~Diening and F.~Gmeineder.
\newblock Sharp trace and korn inequalities for differential operators.
\newblock {\em Potential Analysis}, 2024.

\bibitem{ebo99b}
F.~Ebobisse.
\newblock Fine properties of functions with bounded deformation and
  applications in variational problems.
\newblock {\em {Ph.D. Thesis}}, 1999.

\bibitem{evans2018measure}
L.~C. Evans.
\newblock {\em Measure theory and fine properties of functions}.
\newblock Routledge, 2018.

\bibitem{gmeineder2017critical}
F.~Gmeineder and B.~Raita.
\newblock On critical {$L^p$}--differentiability of {$BD$}--maps.
\newblock {\em Rev. Mat. Iberoam}, 5, 2017.

\bibitem{GmeinRaita2019}
F.~Gmeineder and B.~Rai{\c{t}}{\u{a}}.
\newblock Embeddings for {$\mathbb{A}$}-weakly differentiable functions on
  domains.
\newblock {\em J. Funct. Anal.}, 277(12):33, 2019.
\newblock Id/No 108278.

\bibitem{Gmeineder1}
F.~Gmeineder, B.~Raiţă, and J.~van Schaftingen.
\newblock On limiting trace inequalities for vectorial differential operators.
\newblock {\em Indiana University Mathematics Journal}, 70(5):2133 – 2176,
  2021.

\bibitem{Gmeineder2}
F.~Gmeineder, B.~Raiţă, and J.~Van~Schaftingen.
\newblock Boundary ellipticity and limiting l1-estimates on halfspaces.
\newblock {\em Advances in Mathematics}, 439, 2024.

\bibitem{kohn1980new}
R.~V. Kohn.
\newblock New estimates for deformations in terms of their strains: I)
  estimates of {W}irtinger type for nonlinear strains; ii) functions whose
  linearized strains are measures.
\newblock 1980.

\bibitem{Maggi}
F.~Maggi.
\newblock {\em Sets of finite perimeter and geometric variational problems},
  volume 135 of {\em Cambridge Studies in Advanced Mathematics}.
\newblock Cambridge University Press, Cambridge, 2012.
\newblock An introduction to geometric measure theory.

\bibitem{MSVG2015}
G.~B. {Maggiani}, R.~{Scala}, and N.~{Van Goethem}.
\newblock {A compatible-incompatible decomposition of symmetric tensors in
  \(L^{p}\) with application to elasticity}.
\newblock {\em {Math. Methods Appl. Sci.}}, 38(18):5217--5230, 2015.

\bibitem{mattila1999geometry}
P.~Mattila.
\newblock {\em Geometry of sets and measures in Euclidean spaces: fractals and
  rectifiability}.
\newblock Number~44. Cambridge university press, 1999.

\bibitem{ornstein1962non}
D.~Ornstein.
\newblock A non-inequality for differential operators in the l 1 norm.
\newblock {\em Archive for Rational Mechanics and Analysis}, 11:40--49, 1962.

\bibitem{Raita1}
B.~Raiţä.
\newblock Critical $l^p$-differentiability of $\mathrm{BV}^\mathcal{A}$-maps
  and canceling operators.
\newblock {\em Transactions of the American Mathematical Society}, 372(10):7297
  – 7326, 2019.

\bibitem{rindler2018calculus}
F.~Rindler.
\newblock {\em Calculus of Variations}.
\newblock Springer, 2018.

\bibitem{SVG2016}
R.~{Scala} and N.~{Van Goethem}.
\newblock {Constraint reaction and the Peach-Koehler force for dislocation
  networks}.
\newblock {\em {Math. Mech. Complex Syst.}}, 4(2):105--138, 2016.

\bibitem{TS2}
R.~Temam and G.~Strang.
\newblock {Functions of bounded deformation.}
\newblock {\em Arch. Rational Mech. Anal.}, 75:7--21, 1980.

\bibitem{vanSchaftingen2013}
J.~Van~Schaftingen.
\newblock Limiting {S}obolev inequalities for vector fields and canceling
  linear differential operators.
\newblock {\em J. Eur. Math. Soc. (JEMS)}, 15(3):877--921, 2013.

\end{thebibliography}
\bibliographystyle{plain}

\end{document}